\documentclass[twoside,12pt]{article}
\usepackage{latexsym,amssymb,theorem, epsfig, color}
\usepackage{amsfonts}
\usepackage{amsmath}
\DeclareMathSymbol{\C}{\mathalpha}{AMSb}{"43}
\DeclareMathSymbol{\R}{\mathalpha}{AMSb}{"52}
\DeclareMathSymbol{\Z}{\mathalpha}{AMSb}{"5A}

\newtheorem{theorem}{Theorem}[section]

\newtheorem{definition}[theorem]{Definition}

\newtheorem{remark}[theorem]{Remark}

\newenvironment{proof}[1][Proof]{\begin{trivlist}
\item[\hskip \labelsep {\bfseries #1}]}{\end{trivlist}}

\newcommand{\qed}{\nobreak \ifvmode \relax \else
      \ifdim\lastskip<1.5em \hskip-\lastskip
      \hskip1.5em plus0em minus0.5em \fi \nobreak
      \vrule height0.75em width0.5em depth0.25em\fi}

\title{Inner product quadratures}
\author{Yu Chen \\
Courant Institute of Mathematical Sciences \\
New York University}
\date{Aug 26, 2011}

\begin{document}
\maketitle

\begin{abstract}
We introduce a $n$-term quadrature to integrate inner products of
$n$ functions, as opposed to a Gaussian quadrature to integrate 
$2n$ functions. We will characterize and provide computataional 
tools to construct the inner product quadrature, and establish 
its connection to the Gaussian quadrature.  
\end{abstract}

\tableofcontents

\section{The inner product quadrature}  \label{sec-1}

We consider three types of $n$-term Gaussian quadratures in this paper,
\begin{itemize}
 \item Type-1: to integrate $2n$ functions in interval $[a,b]$.
 \item Type-2: to integrate $n^2$ inner products of $n$ functions.
 \item Type-3: to integrate $n$ functions against $n$ weights.
\end{itemize}
For these quadratures, the weight functions $u$ are not required
positive definite. Type-1 is the classical Guassian quadrature, Type-2 
is the inner product quadrature, and Type-3 finds applications in
imaging and sensing, and discretization of integral equations. 

In this section we will introduce and characterize the Type-2, inner 
product quadrature. \S \ref{sec-2} presents algorithms for constructing 
the quadrature. \S \ref{sec-3} establishes framework to link the first 
two types, and introduce the Type-3 quadrature.  \S \ref{sec-4} illustrates 
our quadrature design method with several examples. \S \ref{sec-5} explores 
generalizations of our quadrature  methods to higher dimensions and 
examines applications to inverse scattering problems.

\subsection{Notation} \label{sec-1.1}

For $x \in [a,b]$ and $k \in [\alpha, \beta]$, by the usual abuse of
notation we will denote by $G(k,x)$ the four related objects
\begin{enumerate}
 \item A family of $L^2$ functions on $[a,b]$, with $k \in 
         [\alpha, \beta]$
 \item The linear subspace spanned by these functions;
 \item The kernel of an integral operator;
 \item The matrix of that operator of size $[\alpha, \beta]$-by-$[a,b]$. 
\end{enumerate}
When $k$ in $G(k,x)$ takes on some finite $n$ values $k_j$ in 
$[\alpha, \beta]$, the resulting set of $n$ functions are denoted by 
\begin{equation}\label{eqn3.10h}
  T(n,x) = \{\; G(k,x), \; x \in [0,1], \quad k = k_j, \; j=1:n \;\}
\end{equation}
Viewed as a matrix of size $n$-by-$[a,b]$, $T(n,x)$ has infinite number 
of columns, and its $n$ rows consist of the $n$ functions $T_j(x)$ in 
$L^2[a,b]$. For example, when $k$ ranges from 0 to $n-1$, the power 
functions
\begin{equation}\label{eqn3.10}
  G(k,x)=\{\; x^k, \quad x \in [0,1], \quad k \in [\alpha, \beta] \;\}
\end{equation}
becomes $T(n,x) = \Pi_n$, polynomials of degree less than $n$. 
Similarly, when $k$ takes on integers, the pure tones
\begin{equation}\label{eqn3.11}
  G(k,x)=\{\; \exp(ikx), \quad x \in [-\pi, \pi], 
                         \quad k \in [-\beta, \beta] \;\}
\end{equation}
becomes trigonometric polynomials. We will first consider a $n$-term 
quadrature to integrate the inner products in $T(n,x)$. 

For $c \in [a,b]$, the column of $T(n,x)$ taken at $x=c$ is the 
$n$-by-1 vector 
\begin{equation}\label{eqn2.2}
  T(n,c) = T(n,x)|_{x=c}
\end{equation}
Given a weight function $u$ we define the dot product in $T(n,x)$ 
\begin{equation}\label{eqn2.1}
  f \cdot g = \int_a^b u(x) f(x) \bar{g}(x) dx
\end{equation}
When $u$ is positive, (\ref{eqn2.1}) will be adopted as the inner 
product for $L^2[a,b]$. 

Given $n$ distinct points $x_j, \; j=1\!:\!n$ in $[a,b]$, let 
$T(n,\{x_j\})$ be the $n$-by-$n$ matrix formed by the $n$ 
columns $T(n,x_j)$, $j=1\!:\!n$. 

Let $B(n,n)$ be the $n$-by-$n$ Gramian matrix of the $n$ functions 
$T_j(x)$ so that
\begin{equation}\label{eqn2.3}
  B(n,n) = T(n,x) \cdot T(x,n)
\end{equation}
where $T(x,n)$ is the complex transpose of $T(n,x)$. Note that the dot 
sign requires inner product (\ref{eqn2.1}) with the underlying weight 
function $u$. For real valued $u$, $B$ is Hermitian.

For positive $u$, let $T^+(x,n)$ be the pseudo inverse of $T(n,x): 
L^2[a,b] \mapsto \C^n$ so that 
\begin{equation}\label{eqn2.4}
  T^+(x,n) = T(x,n) B^{-1}(n,n);
\end{equation}
thus $P_n: L^2[a,b] \mapsto L^2[a,b]$ defined by the formula 
\begin{equation}\label{eqn2.4a}
  P_n(x,y) = T^+(x,n) T(n,y)
\end{equation}
is the orthogonal projector onto $T(n,x)$

\subsection{A $n$-term quadrature for inner products} \label{sec-1.2}
For a positive $u$, let $Q(n,x)$ denote $n$ orthogonal basis functions 
for $T(n,x)$. 
\begin{definition}  \label{def-0.1}
A $n$-term quadrature $\{ x_j, w_j \}$ is one with distinct $x_j$ 
in $[a,b]$ and nonzero $w_j$, $j=1\!:\!n$. 
\end{definition}
\begin{theorem}  \label{thrm-0.1}
({\sc Duality of row and column orthogonalities}) Let the weight 
$u$ of (\ref{eqn2.1}) be positive. There is a $n$-term quadrature 
$\{ x_j, w_j \}$ to integrate all inner products in $T(n,x)$ if 
and only if the $n$ columns of the $n$-by-$n$ matrix $Q(n,\{x_j\})$ 
are orthogonal.
\end{theorem}
\begin{proof}
Obviously, the $n$-term quadrature $\{ x_j, w_j \}$, if exists, 
integrates the Gramian $Q(n,x) \cdot Q(x,n) = I$ with
positive weights $w_j$, namely
\begin{equation}\label{eqn2.25}
  I = Q(n,\{x_j\}) \mbox{\bf diag}\{w_j\} Q(\{x_j\},n)
\end{equation}
so the $n$-by-$n$ matrix $Q(n,\{x_j\}) \mbox{\bf diag}\{\sqrt{w_j}\}$ is 
unitary and thus the columns of $Q(n,\{x_j\})$ are orthogonal.

Now assume that the columns of $Q(n,\{x_j\})$ are orthogonal. 
Let the norm of the j-th column be $1/\sqrt{w_j}$ so that 
$Q(n,\{x_j\}) \mbox{\bf diag}\{\sqrt{w_j}\}$ is unitary, which implies 
that (\ref{eqn2.25}) holds, namely there is a $n$-term quadrature 
$\{ x_j, w_j \}$ to integrate the Gramian $Q(n,x) \cdot Q(x,n) 
= I$; therefore it integrates all inner products in $T(n,x)$.
\qed
\end{proof}

\section{Construct the inner product quadrature}  \label{sec-2}

It is a difficult, nonlinear problem to select $n$ orthogonal columns 
out of infinite number of columns of matrix $Q$. The selection process
can be made a lot easier if the $n$-term quadrature is requried to do
a bit more. In addition to $B$, if the quadrature also integrates 
another Gramian
\begin{equation}\label{eqn2.5}
  A(n,n) = T(n,x) \cdot \mu(x) T(x,n)
\end{equation}
where $\mu$ is a simple function, then the quadrature nodes $x_j$ will 
be recorded in $\mu$. In fact, $\mu(x_j)$ will be eigenvalues of the 
quotient matrix $A B^{-1}$. We first formulate this fact in \S 
\ref{sec-2.1} for polynomial $T(x,n)$. The general case is treated in 
\S \ref{sec-2.2}.

\subsection{Polynomial case} \label{sec-2.1}
Let $T(n,x)$ be the $n$ dimensional space $\Pi_n$ for polynomials 
of degree less than $n$. Let $\mu(x) = x$ so that $A(n,n)$ is
given by
\begin{equation}\label{eqn2.5a}
  A(n,n) = T(n,x) \cdot x T(x,n)
\end{equation}
\begin{theorem}  \label{thrm-1}
If there is a $n$-term quadrature $\{ x_j, w_j \}$ to integrate 
the Gramians $A$ and $B$, then the nodes $x_j$ are eigenvalues of 
$A B^{-1}$
\begin{equation}\label{eqn2.6}
  x_j = \lambda_j(A B^{-1}), \quad j=1\!:\!n  
\end{equation}
provided that $B$ is invertible (when $u$ is not positive definite).
\end{theorem}
\begin{proof}
The $n$-term quadrature exact for $B$ is of the form
\begin{equation}\label{eqn2.6m}
  B = T(n,\{x_j\}) \mbox{\bf diag}\{w_j\} T(\{x_j\},n)
\end{equation}
It follows from invertibility of $B$ that the quadrature weights $w$ 
have no vanishing entry and $T(n,\{x_j\})$ is invertible; therefore,
\begin{eqnarray} \label{eqn2.36}
  AB^{-1} &=& [T(n,\{x_j\}) \mbox{\bf diag}\{w_jx_j\} T(\{x_j\},n)] 
            [T(n,\{x_j\}) \mbox{\bf diag}\{w_j\} T(\{x_j\},n)]^{-1}  \\
          &=& T(n,\{x_j\}) \mbox{\bf diag}\{x_j\} [T(n,\{x_j\})]^{-1}  
\nonumber  
\end{eqnarray}
Thus, the j-th eigenvalue of $AB^{-1}$ is $x_j$ with the
eigenvector $T(n,\{x_j\})$.   \qed
\end{proof}
If the weight function $u$ is positive, then by the proof of Theorem
\ref{thrm-0.1} the quadrature weights are 
\begin{equation}\label{eqn2.7a}
  w_j = \frac{1}{\|Q(n,x_j) \|_2^2}
\end{equation}
If $u$ is not positive definite, since the quadrature is exact for 
the first column of $B$, $w_j$ will be determined by solution of 
the $n$ linear equations for $w$
\begin{equation}\label{eqn2.7}
  T(n,\{x_j\}) \mbox{\bf diag}\{w_j\} T(\{x_j\},1) = T(n,x) \cdot T(x,1)
\end{equation}
where $T(\{x_j\},1)$ is the first column of $T(\{x_j\},n)$, and
$T(x,1)=T_1(x)$ is the first column of $T(x,n)$.
\begin{theorem}  \label{thrm-2}
Let the weight function $u$ be positive definite, and let $v_j$ 
denote the j-th eigenvector of a matrix. There is a $n$-term
quadrature $\{ x_j, w_j \}$ to integrate the Gramian matrices 
$A$ and $B$ if and only if 
\begin{eqnarray}
  \lambda_j(A B^{-1}) &=& x_j, \quad j=1\!:\!n \label{eqn2.8} \\
  v_j(A B^{-1}) &=& T(n,x_j),  \quad j=1\!:\!n \label{eqn2.9}
\end{eqnarray}
\end{theorem}
\begin{proof}
Only need to consider orthonormal basis $T(n,x)$ for which 
$AB^{-1}=A$ is Hermitian with orthogonal eigenvectors. Hence the
proofs of Theorems \ref{thrm-1} and \ref{thrm-0.1} can be adopted
to establish necessity and sufficiency of (\ref{eqn2.8}),
(\ref{eqn2.9})  \qed
\end{proof}
Obviously, the $n$-term quadrature $\{ x_j, w_j \}$ integrating 
the two Gramians integrates all polynomials of degree less than 
$2n$. Thus, Type-1 and Type-2 quadratures are the same for the
polynomial case.

\subsection{Arbitrary functions} 
\label{sec-2.2}

In this section we will characterize and construct an inner product 
quadrature for a set of arbitrary functions $T(n,x)$. 
\begin{theorem}  \label{thrm-5}
Let $B$ be invertible, and let $\lambda_j, v_j$ denote the j-th 
eigenvalue and vector of a matrix. If there is a $n$-term 
quadrature $\{ x_j, w_j \}$ to integrate the Gramians $A$ of 
(\ref{eqn2.5}) and $B$, then
\begin{eqnarray}
  \lambda_j(A B^{-1}) &=& \mu(x_j), \quad j=1\!:\!n \label{eqn2.6c} \\
  v_j(A B^{-1}) &=& T(n,x_j),  \quad j=1\!:\!n \label{eqn2.6d}
\end{eqnarray}
\end{theorem}
The proof is identical to that of Theorem \ref{thrm-1}. 
\begin{definition}  \label{def-1}
A function $\mu$ is said to be a minimal function of $T(n,x)$ if 
\begin{equation}\label{eqn2.10a}
  r(T,\mu) =: \mbox{\bf rank}\{ [(I-P_n) \mu(x)T(x,n) ] \} = 1
\end{equation}
\end{definition}
In other words, the $n$ functions $\mu(x)T(x,n)$ don't entirely lie in 
the span of $T(x,n)$, but the part of $\mu(x)T(x,n)$ that is outside of 
$T(x,n)$ is required to be minimal - the residual dimension $r(T,\mu)$ 
is 1. For example, $\mu=\exp(ix)$ is a minimal function for the subspace
\begin{equation}
  E_m = \mbox{\bf span}[\exp(ikx), \; k=-m:m],
            \quad m>0, \; x \in [-\pi, \pi]  \label{eqn3.3b}
\end{equation}
More general definition for minimal function will be given in
\S \ref{sec-3.2}. Modifications are also necessary for higher 
dimensions.
\begin{theorem}  \label{thrm-6}
Let the weight function $u$ be positive definite, and let $v_j$ 
denote the j-th eigenvector of a matrix. The three conditions are 
equivalent \\
(i) There is a $n$-term quadrature $\{ x_j, w_j \}$ to integrate 
$A$ and $B$ \\
(ii) The quotient matrix $A B^{-1}$ is diagonalizable with
\begin{eqnarray}
  \lambda_j(A B^{-1}) &=& \mu(x_j), \quad j=1\!:\!n \label{eqn2.8c} \\
  v_j(A B^{-1}) &=& T(n,x_j),  \quad j=1\!:\!n \label{eqn2.9d}
\end{eqnarray}
(iii) There exist such $\{ x_j \}$ that for every $p_n \in 
\mbox{\bf span}[(I-P_n) \mu(x) T(n,x)]$
\begin{equation}\label{eqn2.11}
  p_n(x_j) = 0, \quad j=1\!:\!n  
\end{equation}
and that the $n$-by-$n$ matrix $T(n,\{x_j\})$ is invertible.
\end{theorem}
\begin{proof}
The proof of equivalency of (i) and (ii) is similar to that of Theorem
\ref{thrm-2}. Now we establish equivalency of (ii) and (iii). By (ii), 
and by (\ref{eqn2.4}) and (\ref{eqn2.4a}), 
\begin{eqnarray}
  \mu(x_j) T(n,x_j) 
&=& A B^{-1} T(n,x_j), \quad j=1\!:\!n \label{eqn2.8h} \\
&=& T(n,x) \cdot \mu(x) T^+(x,n) T(n,x_j)) \nonumber \\
&=& [T(n,x)\mu(x)] \cdot P_n(x,x_j) \nonumber \\
&=& \{ P_n [T(n,x)\mu(x)] \}_{x=x_j}  \nonumber \\
&=& \{ (I-I+P_n) [T(n,x)\mu(x)] \}_{x=x_j}  \nonumber \\
&=& T(n,x_j)\mu(x_j) - \{ (I-P_n) [T(n,x)\mu(x)] \}_{x=x_j}  \nonumber 
\end{eqnarray}
which holds if and only if 
\begin{equation}\label{eqn2.11a}
  \{(I-P_n) [\mu(x) T(n,x)] \}(x_j) = 0, \quad j=1\!:\!n,  
\end{equation}
namely (\ref{eqn2.11}) holds, and the $n$-by-$n$ matrix $T(n,\{x_j\})$ 
is invertible. \qed
\end{proof}
Theorem \ref{thrm-6} does not require that $\mu$ is minimal, but if
it is not then all $p_n \in \mbox{\bf span}[(I-P_n) \mu(x) T(n,x)]$
must share $n$ common roots at the quadrature nodes ${x_j}$. In other
words, it is unlikely for a $n$-term quadrature to integrate both
$B$ and $A$ exactly if $\mu$ is not minimal.
 
If $T(n,x)=\Pi_n$, polynomials of degree less than $n$, then 
$\mu(x) = \alpha x + \beta$, $\alpha \not =0$ is a minimal function 
whereas $x^2$ is not. With a minimal $\mu$, $p_n$ of (\ref{eqn2.11}) 
is the orthogonal polynomial of degree $n$, provided that the weight 
$u$ is positive definite. The condition (\ref{eqn2.11}) is well known 
as a part of the Gauss formula. 

However, when the functions $T(n,x)$ are not polynomials, the $n$-term
quadrature formula may not be a Gaussian quadrature in the classical
sense. In general, to integrate the inner products in $A$ and $B$ is
not the same as to integrate some $2n$ functions. Conversely, given 
a set of $2n$ functions to integrate by a Gaussian quadrature, additional 
work is required to reformulate this Type-1 quadrature as a Type-2, 
inner product quadrature. This issue will be addressed in the next 
section.

\section{Product law and minimal functions}  \label{sec-3}

In this section we will establish framework for converting the Type-1
quadrature to Type-2, inner product quadrature. While an inner 
product quadrature is natural in its own right and immediately useful 
in many applications, other applications require Type-1 quadratures. 
For many familiar and widely used families of functions the two 
quadrature problems turn out to be equivalent or nearly so. We 
introduce the notion of factor space in \S \ref{sec-3.1} and 
minimal function in \S \ref{sec-3.2} to connect the two types of
quadratures.

\subsection{Factor spaces} \label{sec-3.1}

In this section we introduce the product law and factor space for 
a given set of $2n$ functions, so as to convert a Type-1 quadrature 
for the $2n$ functions to a Type-2 for the Gramian matrix of the 
factor space. 

Let the rows of $V(2n,x): L^2[a,b] \mapsto \C^{2n}$ consist of
a set of $2n$ linearly independent functions, which span a linear 
subspace of $L^2[a,b]$ denoted also by $V(2n,x)$.
\begin{definition}  \label{def-2}
The linear subspace $V(2n,x)$ is said to have a factor space 
$T(n,x)$ with a multiplier $\mu$ if 
\begin{equation}\label{eqn3.2}
  \mbox{\bf span}\{T_i(x) \bar{T}_j(x), \; T_i(x) \mu(x) \bar{T}_j(x), \; 
         1\leq i,j \leq n \} = V(2n,x)
\end{equation}
\end{definition}
As an example, the linear space $\Pi_{2n}$ for polynomials of degree
less than $2n$ has a factor space $\Pi_n$ with $\mu(x)=x$ as the 
multiplier. Likewise, let
\begin{equation}\label{eqn3.3a}
  G_m = \mbox{\bf span}[1, \sin(jx), \cos(jx), \; j=1\!:\!m-1], 
            \quad m>1 
\end{equation}
Then $G_m$ is a factor space of $G_{2m}$ with $\mu(x)=\cos(x)$,
whereas $E_m$ of (\ref{eqn3.3b}) is a factor space of $E_{2m}$ 
with $\mu=1$. 

Obviously, a quadrature integrating the inner products in the 
factor space will also integrate the functions in $V(2n,x)$. 
In this respect, the notion of a factor space can be relaxed in two 
directions (i) Let the span in (\ref{eqn3.2}) include, rather
than equal to, $V(2n,x)$ (ii) Let the span in (\ref{eqn3.2}) 
approximate $V(2n,x)$ to a given precision.
\begin{definition}  \label{def-3}
The linear subspace $V(2n,x)$ is said to obey the product 
law if there exist $n$ functions $T(n,x)$ such that $V(2n,x)$ is
a subspace of the product space
\begin{equation}\label{eqn3.4}
  \Pi(T) = \mbox{\bf span}\{T_i(x) \bar{T}_j(x), \; 1\leq i,j \leq n \}  
\end{equation}
A linear subspace $V(2n,x)$ is said to obey the product law
to precision $\epsilon>0$ if for any $f \in V(2n,x)$ the distance 
between $f$ and $\Pi(T)$ is $\epsilon$.
\end{definition}
As an example, by Neumann's addition formula 9.1.78 of \cite{Redtable}, 
\begin{equation}\label{eqn3.6}
  J_m(x) =  \sum_{k=0}^m J_k(x/2) J_{m-k}(x/2) + 
            2\sum_{k=1}^\infty (-1)^k J_k(x/2) J_{m+k}(x/2) 
\end{equation}
For $x \in [0,b]$, and for a prescribed precision $\epsilon>0$, 
there exists $\delta>0$ so that 
\begin{equation}\label{eqn3.7}
  | J_s(x/2) | < \epsilon, \quad s>b+\delta
\end{equation}
Thus, to precision $O(\epsilon)$, only finite number of terms in 
(\ref{eqn3.6}) remain: $J_s(x/2)$ for $0\leq s \leq b+\delta$. 
Consequently, the space
\begin{equation}\label{eqn3.8}
  \mathcal{V} = \mbox{\bf span}\{J_m(x), \; 0\leq m \leq 2(b+\delta)\} 
\end{equation}
obeys the product law to precision $O(\epsilon)$, and a quadrature 
integrating the inner products in 
\begin{equation}\label{eqn3.9}
  \mathcal{T} = \mbox{\bf span}\{J_s(x/2), \; 0\leq s \leq b+\delta\} 
\end{equation}
exactly or to precision $O(\epsilon)$ will also integrate functions
in $\mathcal{V}$ to precision $O(\epsilon)$. 

Factor space and product law can also be extended to a family of 
infinite number of functions, denoted by $G(k,x)$, $x \in [a,b]$, 
with $k \in [\alpha, \beta]$ the family parameter.
\begin{definition}  \label{def-3.5}
The family of functions $G(k,x)$ obeys the product law if there exist 
two families of functions $T(k,x)$, $S(\kappa,x)$, $x \in [a,b]$, 
$k \in [\alpha_1, \beta_1]$, $\kappa \in [\alpha_2, \beta_2]$ such 
that $G(k,x)$ is a subset of the product space
\begin{equation}\label{eqn3.4x}
  \Pi(S,T) = \mbox{\bf span}\{T(k,x) S(\kappa,x), \; 
             k\in [\alpha_1, \beta_1], \; \kappa \in [\alpha_2, \beta_2] \}
\end{equation}
Moreover, $G(k,x)$ is said to have factor spaces $T(k,x)$ and 
$S(k,x)$ if
\begin{equation}\label{eqn3.2x}
  \mbox{\bf span}\{G(k,x), \; \alpha \leq k \leq \beta \} = \Pi(S,T) 
\end{equation}
Finally, $G(k,x)$ is said to have a factor space $T(k,x)$ if it has
the factor spaces $T(k,x)$ and $S(k,x)$ with $S(k,x)=\bar{T}(k,x)$. 
\end{definition}
Accordingly, the $n$-by-$n$ Gramians $B$ of (\ref{eqn2.3}) and $A$ of 
(\ref{eqn2.5}) can be extended to operator case.
\begin{definition}  \label{def-3.6}
Given a function $\mu(x)$, the linear operators defined by
\begin{eqnarray}
  A(k,k') &=& T(k,x) \cdot \mu(x) S(x,k') \label{eqn3.44} \\
  B(k,k') &=& T(k,x) \cdot        S(x,k') \label{eqn3.45}
\end{eqnarray}
are referred to as the Gramians associated with the factor spaces 
$T(k,x)$ and $S(k,x)$.
\end{definition}
\begin{theorem}  \label{thrm-3.3}
Let the $m$-by-$n$ matrices
\begin{eqnarray}
  A(m,n) &=& T(m,x) \cdot \mu(x) S(x,n) \label{eqn3.44a} \\
  B(m,n) &=& T(m,x) \cdot        S(x,n) \label{eqn3.45b}
\end{eqnarray}
be the Gramians associated with the $m$-by-$[a,b]$ matrix $T(m,x)$ 
and the $n$-by-$[a,b]$ matrix $S(n,x)$ and a scalar function $\mu$.
Let the rank of $B$ be $r$. If there is a $r$-term quadrature 
$\{ x_j, w_j \; j=1\!:\!r \}$ precise for $A$ and $B$, then the $m$-by-$m$ 
matrix $A B^+$ has $r$ eigenvalues and corresponding eigenvectors of 
the form
\begin{eqnarray}
  \lambda_j(A B^+) &=& \mu(x_j), \quad j=1\!:\!r \label{eqn3.47} \\
  v_j(A B^+) &=& T(m,x_j),       \quad j=1\!:\!r \label{eqn3.48}
\end{eqnarray}
The remaining $m-r$ eigenvalues are zero.
\end{theorem}
\begin{proof}
The proof is similar to that of Theorem \ref{thrm-1}. Since $B$ is 
of rank $r$, the existence of the $r$-term quadrature, exact for $B$, 
implies that the quadrature weights $w$ have no vanishing entry and 
the $m$-by-$r$ matrix $T_r = T(m,\{x_j\})$ and the $r$-by-$n$ matrix 
$S_r = S(\{x_j\},n)$ are both full rank; therefore, 
\begin{equation}\label{eqn3.49}
  B^+ = [T_r \, \mbox{\bf diag}\{w_j\} \, S_r]^+ = S_r^+ \, \mbox{\bf diag}\{w_j\}^{-1} \, T_r^+
\end{equation}
Using $S_r S_r^+ = I$ we have
\begin{eqnarray} \label{eqn3.50}
  AB^+ &=& T_r \, \{\mu(x_j)\} \, \mbox{\bf diag}\{w_j\} \, S_r \; 
           S_r^+ \mbox{\bf diag}\{w_j\}^{-1} T_r^+   \\
       &=& T_r \, \{\mu(x_j)\} \, T_r^+  \nonumber  
\end{eqnarray}
It follows immediately that the $m$-by-$m$ square matrix $AB^+$, 
being of rank $r$ or less, will have $m-r$ zero eigenvalues, and 
owing to $T_r^+ T_r = I$ the $r$ remaining eigenvalues and vectors 
are given by (\ref{eqn3.47}), (\ref{eqn3.48}).   \qed
\end{proof}
Theorem \ref{thrm-1} is a special case of Theorem \ref{thrm-5} which
is a special case of Theorem \ref{thrm-3.3}. For quadrature design, 
we are only interested in eigenvectors, if exist, of the form 
$T(m,x_j)$.
\begin{definition}  \label{def-3.7}
The eigenvectors of $AB^+$ of the form $T(m,x_j)$ are referred to 
as the position eigenvectors.
\end{definition}
The existence of the position eigenvectors is necessary for that 
of a Gaussian quadrature. The next theorem, straightforward to 
verify, says that the eigenvalues for the quotient matrix is 
invariant under the change of bases by (\ref{eqn3.50a}), 
(\ref{eqn3.50b}). 
\begin{theorem}  \label{thrm-3.3a}
Let the square matrices $t(m,m)$ and $s(n,n)$ be invertible. Let 
the change of bases, from $T(m,x)$ to $\tilde{T}(m,x)$, and from
$S(n,x)$ to $\tilde{S}(n,x)$, be defined by
\begin{eqnarray} 
  \tilde{T}(m,x) &=& t(m,m) T(m,x) \label{eqn3.50a}\\
  \tilde{S}(n,x) &=& s(n,n) S(n,x) \label{eqn3.50b}
\end{eqnarray}
Then the two quotient matrices $AB^+$ associated with the old and
new bases are similar, with $t(m,m)$ as the similarity transform. 
\end{theorem}

\subsection{Minimal functions} \label{sec-3.2}

Minimal function was defined in \S \ref{sec-2.2} for a set of
$n$ functions $T(n,x)$. In this section, we will introduce 
minimal function for a family of infinite number of functions 
$G(k,x)$. 
\begin{definition}  \label{def-5}
{\sc (Informal)}  A method to grow a family of functions $G$ is to 
multiply the existing family members by a function $\mu$, which 
may not be in the family. The resulting functions are linearly 
combined with those in the family to generate a new function. The 
function $\mu$, with proper normalization, is the minimal function. 
\end{definition}
Typical 3-term recursions use this scheme to generate a class 
of functions. For example, the Bessel functions require $\mu(x) = 1/x$ 
as the multiplier to push the family one step forward, or backward. 

For a precise definition of minimal function, let the new function 
$G(\beta+h,x)$ be generated by linear combination of $\mu(x) G(k,x)$ 
and $G(k,x)$ over $k \in [\alpha, \beta]$. We scale $\mu$ such that 
it appears in the linear combination as follows 
\begin{equation}\label{eqn3.12}
  G(\beta+h, x) = h G(\beta,x)\mu(x) + G(\beta,x) + \mbox{tail}
\end{equation}
The tail vanishes as $h \rightarrow 0$, provided that $\mu$ is the log 
derivative of $G$ with respect to $k$.
\begin{definition}  \label{def-6}
Let $G(k,x)$ be differentiable with respect to $k$ in $[\alpha,
\beta]$ for almost every $x \in [a,b]$. The function 
\begin{equation}\label{eqn3.13}
  \mu(x,k)|_{k=\beta} = \left\{\frac{\partial}{\partial k} 
               \log G(k,x)\right\}_{k=\beta}
\end{equation}
is referred to as a specific minimal function of $G(k,x)$ at $k=\beta$.
If $\mu(x,k)$ is independent of $k$ or if the dependence is separable
\begin{equation}\label{eqn3.13a}
  \mu(x,k) = p(k) q(x) \quad \mbox{so that} \quad 
          \partial_k G(k,x) = p(k)\, G(k,x)\, q(x)
\end{equation}
then it is referred to as the (general) minimal function of $G(k,x)$.
\end{definition}
By (\ref{eqn3.13}), the minimal functions for the power functions
(\ref{eqn3.10}) and exponentials (\ref{eqn3.11}) are $\log(x)$ and
$x$. By $\mu$'s dependence on $k$, we divide $G$ into three varieties
\begin{description}
 \item (V.1) It is independent of $k$. 
 \item (V.2) The dependence is separable.
 \item (V.3) The dependence is not separable.
\end{description}
There are two cases for constructing a quadrature, whether Type-1 or 2 
\begin{description}
 \item (C.1) Design a quadrature with a given weight $u$. 
 \item (C.2) Design a quadrature without $u$ given explicitly.
\end{description}
(C.1) is typical of quadrature design for numerical integration; 
the weight $u$ is given explicitly. (C.2) arises from certain 
applications such as inverse problems or signal processing where 
the measurement or signal is the exact integrals 
\begin{equation}\label{eqn1a}
    s(k) = \int_a^b G(k,x) u(x) dx, \quad k \in [\alpha, \beta]
\end{equation}
with an underlying, fixed, but unknown $u$.
 
For (C.2), the only data available for Type-1 quadrature design is 
$s(k)$. When reformulated as a Type-2 quadrature problem, (V.1) and 
(V.2), not (V.3), will be useful in constructing the Gramians $A$ 
and $B$ out of the data $s(k)$. The procedures for constructing the 
Gramians by (V.1) and (V.2) are so similar that in the sequel we will 
only consider (V.1), namely (V.2) with $p(k)\equiv 1$.

For (C.1), the weight function $u$ is given and the Gramians can 
be constructed directly by their definitions (\ref{eqn3.44}) and
(\ref{eqn3.45}) for a Type-2 quadrature, or for a Type-1 quadrature
provided that $G$ has factor spaces $T(k,x)$ and $S(k,x)$.
(V.3) will be useful for (C.1).   

A specific minimal function exists for an arbitrary system of 
functions $G(k,x)$. In contrast, only certain function classes
have (general) minimal functions. The next theorem is a direct
consequence of Definition \ref{def-6}. 
\begin{theorem}  \label{thrm-3.4}
$G(k,x)$ has a minimal function if and only if 
\begin{equation}\label{eqn3.13b}
  G(k,x) = \exp(p(k)q(x)) r(x) 
\end{equation}
Moreover, if $G(k,x)$ has a minimal function then it has a 
factor space
\begin{equation}\label{eqn3.13c}
  T(k,x) = [G(k,x)]^{1/2} = \exp(p(k)q(x)/2) \sqrt{r(x)} 
\end{equation}
\end{theorem}
For example, the family $x^k = \exp(k \log(x))$ is of this 
exponential type. The family $k^x = \exp(\log(k) x)$ is also 
of this type.

When $k$ takes only on discrete values, say integers, the differential 
form (\ref{eqn3.13}) for $\mu$ can be replaced by a finite difference 
for certain classes of functions, among them are polynomials and modified 
Bessel functions:  
\begin{eqnarray}
  \mu(x) &=& \frac{x^{n+1} - x^n}{1 \cdot x^n} = x-1 \label{eqn3.14} \\
  \mu(x) &=& \frac{I_{n+1}(x) - I_{n-1}(x)}{2 \cdot I_n(x)} = -n/x   
             \label{eqn3.15}
\end{eqnarray}
where $x-1$ can be normalized to $x$, and the separable 
$\mu(x,n)=-n/x$ to $1/x$.

For certain applications $k \in [\alpha, \beta]$ is restricted on a 
uniform grid with step size $h$, so that only a finite number of the 
family members $G(\alpha+jh,x)$, $j=0:n$, $nh=(\beta-\alpha)$, are to 
be integrated. For example,
\begin{eqnarray}
  \mu(x) &=& x^h, \quad \mbox{for power functions}    \label{eqn3.16} \\
  \mu(x) &=& e^{ihx}, \quad \mbox{for exponentials}   \label{eqn3.17} \\
  \mu(x) &=& \cos(x), \quad \mbox{for trigonometrics} \label{eqn3.18}
\end{eqnarray}
are appropriate minimal functions. 
\begin{remark}  \label{rmk-3.2}
The minimal function introduced in Definition \ref{def-6} is asscociated 
with differentiation with repect to $k$. In general, $\partial_k$ in 
(\ref{eqn3.13a}) is replaced by a map $L_k$ which operates on $G(k,x)$
as a family of functions of $k$. For example, if Sturm-Liouville equation
$(-L_x+k^2)u(x)=0$ has a solution of the form $u=G(k,x)=G(kx)$, then 
interchanging the roles of $k$ and $x$ we have $L_kG(kx) = x^2 G(kx)$.
In other words, the family of functions $G(k,x)=G(kx)$ has a minimal
function $x^2$ with respect to the operator $L_k$.
\end{remark}

\subsection{Fold data into Gramians - signal processing} 
\label{sec-3.4}

In quadrature design for numerical integration, the weight function 
$u$ is usually prescribed. For other applications, such as optimal 
design or inverse problems, $u$ is either a variable or not given 
explicitly. 

When $u$ is not given and the exact integrals $s(k)$ of (\ref{eqn1a}) 
is the only available data, the first step in quadrature design for 
$G(k,x)$ is to process the signal $s$ to construct the Gramians $A$ 
and $B$.
 
In this section, we will describe the signal processing operations 
for converting Type-1 quadrature for $G$ to Type-2 for the Gramians.
This signal processing is not required to contruct the Gramians if 
$u$ is available. 

Let $G(k,x)$ have a factor space $T(k,x)$, $k\in [\alpha_1, \beta_1]$ 
and a minimal function $\mu$ so that 
\begin{equation}\label{eqn3.20}
   \mbox{\bf span}\{G(k,x), \; \alpha \leq k \leq \beta \}  
 =  \mbox{\bf span}\{T(k,x) \bar{T}(k',x), \; 
         \alpha_1 \leq k,k' \leq \beta_1 \}
\end{equation}
and 
\begin{equation}\label{eqn1b}
    s'(k) = \int_a^b \mu(x) G(k,x) u(x) dx
\end{equation}
with the latter obtained by differentiating (\ref{eqn1a}). By 
(\ref{eqn3.20}), there exists linear combination coefficients 
$F(k,k',\kappa)$ to reproduce $T(k,x) \bar{T}(k',x)$ as a linear
combination of $G(k,x)$:
\begin{equation}\label{eqn3.21}
    T(k,x) \bar{T}(k',x) = \int_\alpha^\beta F(k,k',\kappa) 
       G(\kappa,x) d\kappa
\end{equation}
Integrating (\ref{eqn3.21}) with respect to $x$ against $\mu(x) u(x)$
we rewrite the result and (\ref{eqn1b}) in matrix form
\begin{eqnarray}
  s'(k) &=& G(k,x) \mu(x)u(x)  \label{eqn3.22} \\
  T(k,x) \cdot \mu(x) T(x,k') &=& F(k,k',\kappa) G(\kappa,x)
      \mu(x)u(x) = F(k,k',\kappa) s'(\kappa) \label{eqn3.23}
\end{eqnarray}
The operator $T(k,x) \cdot \mu(x) T(x,k')$, by (\ref{eqn3.44}), is 
Gramian matrix $A(k,k')$, and (\ref{eqn3.23}) shows that the 
derivative of the signal is required to construct $A$, and that
how the vector $s'$ is packed into $A$ by the folding operator $F$. 
The other Gramian matrix $B$ is also constructed by the same 
folding process performed on the signal $s$
\begin{eqnarray}
  A(k,k') &=&  F(k,k',\kappa) s'(\kappa) \label{eqn3.24} \\
  B(k,k') &=&  F(k,k',\kappa) s (\kappa) \label{eqn3.25}
\end{eqnarray}
As an example, let $G(k,x)$ be exponentials defined by (\ref{eqn3.11}), 
which has a factor space 
\begin{equation}\label{eqn3.26}
  T(k,x)=\{\; \exp(ikx), \quad x \in [-\pi, \pi], 
                         \quad k \in [-\beta/2, \beta/2] \;\}
\end{equation}
By (\ref{eqn3.21}), the folding operator is
\begin{equation}\label{eqn3.27}
  F(k,k',\kappa)=\delta(k-k'-\kappa), \quad k,k' \in 
     [-\beta/2, \beta/2], \quad \kappa \in [-\beta, \beta] \;\}
\end{equation}
so that for a fixed $\kappa$, the kernel $F(k,k',\kappa)$ is zero 
everywhere except on the diagonal $k-k'=\kappa$; the Gramians $A$ 
and $B$ of (\ref{eqn3.24}) and (\ref{eqn3.25}) are Toeplitz matrices 
with $s'(\kappa)$ and $s(\kappa)$ on the diagonal $k-k'=\kappa$. 

Folding a data vector, or signal, into a matrix or matrices and 
subsequently processing them is inherently a data analysis 
procedure. When $G(k,x)$ and its factor space $T(k,x)$ share the 
same minimal function $\mu$, the Gramian matrices $A$ and $B$ are 
of the form
\begin{eqnarray}
  B(k,k') &=&  T(k,x) \cdot T(x,k') \label{eqn3.36} \\
  A(k,k') &=&  T(k,x) \cdot \mu(x) T(x,k')  \\
          &=&  [\partial_k T(k,x)] \cdot T(x,k') 
           = \partial_k B(k,k') \label{eqn3.37}
\end{eqnarray}
\begin{theorem}  \label{thrm-11}
Suppose that $G(k,x)$ and its factor space $T(k,x)$ share the same 
minimal function $\mu$, and that the weight function $u$ of
(\ref{eqn2.1}) is nonzero almost everywhere on $[a,b]$. Then the 
quotient matrix $Q=AB^{-1}$ is the differential operator, with 
respect to $k$, restricted on the subspace $T(k,x)$.
\end{theorem}
If $u$ vanishes on a subset of $[a,b]$ of positive measure, the 
quotient matrix $Q$ will still be a differential operator 
restricted on the range space of $B$, which is a subspace of 
$T(k,x)$.

\subsection{Regularization} \label{sec-3.5}

Once the two Gramians are constructed, there are two issues with 
computing the quotient matrix $Q = AB^{-1}$ (i) Inverting the 
compact operator $B$ (ii) For a prescribed precision $\epsilon>0$, 
replace $Q$ by a finite, $n$-by-$n$ square matrix for subsequent eigen 
decomposition. The two issues can be tackled together by 
regularization of $A, B$: Approximate $A, B$ with finite rank 
operators $A_n, B_n$. 

Ideally, we should find a function $s_n(k)$ to approximate the 
data $s(k)$ in a least squares sense to the prescribed precision 
which when packed by (\ref{eqn3.25}) gives rise to $B_n$ of rank 
$n$. Solving such a nonlinear least squares problem is not known 
to be tractable in cost or convergence, so suboptimal schemes are 
sought instead. One of them requires SVD on $B$ with $\epsilon$
as the cut off precision to construct a rank $n$ best approximation 
to $B$, so that 
\begin{eqnarray}
  B(k,k') &\approx& U(k,n) \Sigma(n,n) V(n,k') \label{eqn3.28} \\
  A(k,k') &\approx& U(k,n) S(n,n)      V(n,k') \label{eqn3.29} \\
  Q(k,k') &\approx& U(k,n) S(n,n)\Sigma^{-1}(n,n) U(n,k')  \label{eqn3.30}
\end{eqnarray}
namely, both $B$ and $A$ are projected on the $n$ dimensional
column (or range) subspace spanned by $U(k,n)$ and row (or domain)
space spanned by $V(k,n)$. Note that while $\Sigma$ is diagonal,
$S$ is generally not. Finally, by Theorem \ref{thrm-5}, a $n$-term
quadrature of finite precision proportional to the prescribed can
be attempted by solving the eigenvalue problem for the projected 
version of $Q$ 
\begin{equation}\label{eqn3.31}
  \tilde{Q}(n,n) = S(n,n)\Sigma^{-1}(n,n)
\end{equation}

\subsection{Type-3 quadratures for integral equations} \label{sec-3.6}

Let $A$, $B$ be the $m$-by-$n$ Gramian matrices of Theorem 
\ref{thrm-3.3}. Let the rank of $B$ be $r$. A $r$-term Type-3 
quadrature uses the nodes $\{ x_j, j=1\!:\!r \}$ and weights 
$W=\{w_{ij} \; i=1\!:\!m, j=1\!:\!r \}$ to integrate $A$ and $B$ 
\begin{eqnarray}
  A(m,n) &=& W(m,j) \mu(x_j) S(x_j,n) \label{eqn3.47a} \\
  B(m,n) &=& W(m,j) S(x_j,n)          \label{eqn3.47b}
\end{eqnarray}
In other words, the $m$ functions $T(m,x)$ of (\ref{eqn3.44}) are
regarded as the weight functions for the Type-3 quadrature. 
\begin{theorem}  \label{thrm-3.5}
If there is a $r$-term quadrature (\ref{eqn3.47a}), (\ref{eqn3.47b}), 
then the $m$-by-$m$ matrix $A B^+$ has $r$ eigenvalues and 
corresponding eigenvectors of the form 
\begin{eqnarray}
  \lambda_j(A B^+) &=& \mu(x_j), \quad j=1\!:\!r \label{eqn3.47c} \\
  v_j(A B^+) &\propto& W(m,j),   \quad j=1\!:\!r \label{eqn3.48c}
\end{eqnarray}
The remaining $m-r$ eigenvalues are zero.
\end{theorem}
The proof is nearly identical to that of Theorem \ref{thrm-3.3}, and 
is omitted. Let 
\begin{equation}\label{eqn1c}
    v(y) = \int_a^b G(y,x) u(x) dx, \quad y \in [a, b]
\end{equation}
be an integral equation for $u$ on $[a, b]$. Let $\{ y_i, i=1\!:\!m \}$
be $m$ points in $[a, b]$. Let $T(m,x)=\{ G(y_i,x), i=1\!:\!m \}$. 
Finally, let $u \in S(x,n)$, namely $u$ is in the span of the $n$ 
functions $S$. Then $W$ of Theorem \ref{thrm-3.5} is a discretization 
of the integral equation  
\begin{equation}\label{eqn1d}
    v(y_i) \approx \sum_{j=1}^r W_{ij} u(x_j) 
\end{equation}
which is precise for $u \in S(x,n)$.

\section{Examples} \label{sec-4}

In this section we present several examples to illustrate our 
quadrature design methods. In \S \ref{sec-4.1} we construct 
quadratures for non-positive definite weight $u$. \S \ref{sec-4.2} 
and \S \ref{sec-4.3} construct quadratures for power and exponential
functions. 


\subsection{Quadratures for non-positive definite weights} 
\label{sec-4.1}

Gaussian quadratures may not exist for non-positive definite weights.
As an example, we consider n-term Gaussian quadratures to integrate 
polynomials of degree less than $2n$, against the weight function
\begin{equation}\label{eqn3.14c}
  u(x)=\sin(3\pi x)
\end{equation}
in $[-1, 1]$. The oddness of $u$ and the optimality of Gaussian
quadrature preclude Gaussian quadratures of odd $n$, otherwise $x=0$ 
must be a quadrature node where $u$ vanishes which makes the node
useless. Not all even $n$ values support a Gaussian quadrature. 
For the weight given by (\ref{eqn3.14c}), there is a Gaussian quadrature
for $n=16$, and $n=18$, but not for $n=14$. Whenever there is a Gaussian 
quadrature, it can be constructed by Theorem \ref{thrm-1}. Figure \ref{fig-4.1}
shows the locations of the quadrature nodes in $[-1,1]$, and the 
quadrature weights. The weights are negative wherever $u$ is negative. 
\begin{figure}[h]
\epsfig{file=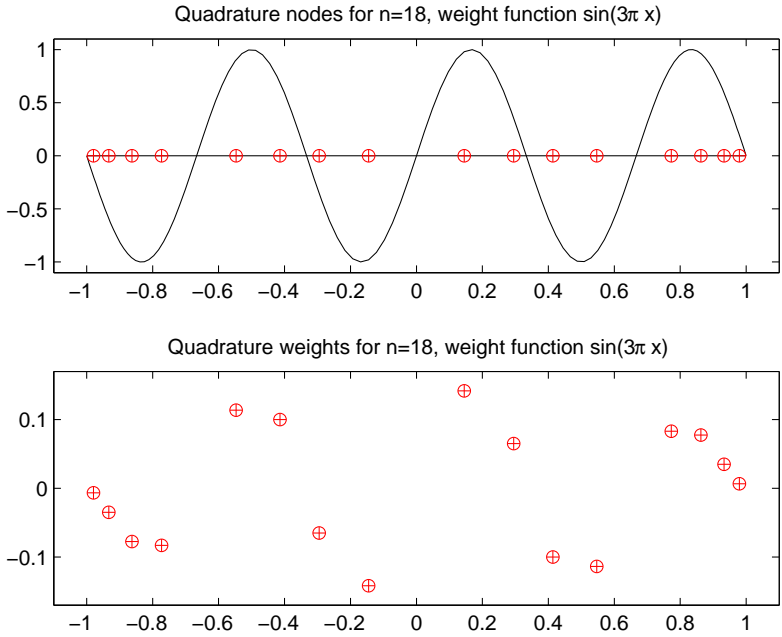, width=4.9in}
\caption{Quadrature nodes and weights for $u(x)=\sin(3\pi x)$}
\label{fig-4.1}
\end{figure}

\subsection{Power functions, Hankel Gramians} \label{sec-4.2}

To integrate the power functions
\begin{equation}\label{eqn3.10a}
  G(k,x)=\{\; x^k, \quad x \in [0,1], \quad k \in [\alpha, \beta] \;\}
\end{equation}
against a weight $u$ with a $n$ term quadrature, we follow \S 
\ref{sec-3.4} to construct Gramian $B$ from the exact integrals 
$s(k)$, and Gramian $A$ from $s'(k)$.

By (\ref{eqn3.13}), the minimal function is $\mu(x)=\log(x)$. The
power functions obey the product law of Definition \ref{def-3.5}, with
\begin{equation}\label{eqn3.10b}
  T(k,x)=\{\; x^k, \quad x \in [0,1], \quad k \in [\alpha/2, \beta/2] \;\}
\end{equation}
By (\ref{eqn3.21}), the folding kernel $F$, cf (\ref{eqn3.27}), is 
\begin{equation}\label{eqn3.27a}
  F(k,k',\kappa)=\delta(k+k'-\kappa), \quad k,k' \in 
     [\alpha/2, \beta/2], \quad \kappa \in [\alpha, \beta] \;\}
\end{equation}
Therefore, the Gramians $A$ and $B$ are Hankel matrices with 
$s'(\kappa)$ and $s(\kappa)$ on their anti-diagonals $k+k'=\kappa$.

For a numerical experiment, we construct a Gaussian quadrature for 
$G(k,x)=x^k$, $x \in [a,b]=[0,1]$, $k \in [\alpha, \beta]=[-1/3, 1/2]$ 
by constructing an inner product quadrature for the factor space 
$T(k,x)= k^x$, $x \in [a,b]=[-3,3]$, $k \in [\alpha/2, \beta/2]
=[-1/6, 1/4]$. Following the procedures of \S \ref{sec-3.5}, a $n$ term 
quadrature, though not precise to integrate all functions in $G(k,x)$,
was constructed from the n-by-n Gramians $A$ and $B$ of (\ref{eqn3.28}.
For a cut off precision $\epsilon=10^{-12}$, $n=9$. Figure \ref{fig-4.2}
shows the locations of nodes in $[-3,3]$, and the relative error of the
quadrature as a function of $k \in [1/16, 4]$. 
\begin{figure}[h]
\epsfig{file=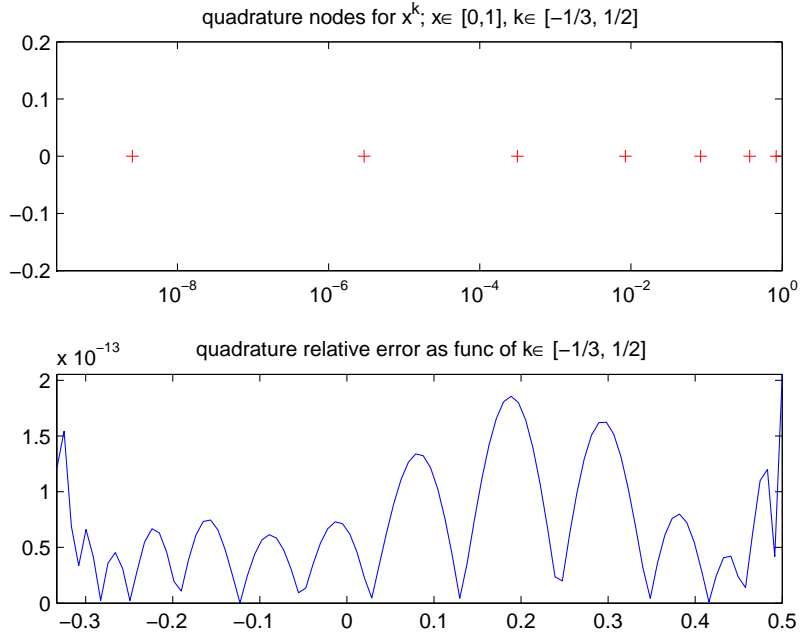, width=4.9in}
\caption{Quadrature nodes and relative error for $G(k,x)=x^k$}
\label{fig-4.2}
\end{figure}

\subsection{Exponentials $k^x$, hyperbolic Gramians} \label{sec-4.3}

This subsection is analogous to the preceding one; therefore, we
will only provide the essentials. The family of exponential functions
\begin{equation}\label{eqn3.10c}
  G(k,x)=\{\; k^x, \quad x \in [a,b], \quad k \in [\alpha, \beta]
           \;\}, \quad \alpha>0
\end{equation}
is not equivalent to $\exp(kx)$. The minimal function is dependent on
$k$ but the dependence is separable 
\begin{equation}\label{eqn3.10d}
  \mu(x,k) = x/k := \mu(x)/k
\end{equation}
The factor space
\begin{equation}\label{eqn3.10e}
  T(k,x)=\{\; k^x, \quad x \in [a,b], \quad k \in 
         [\sqrt{\alpha}, \sqrt{\beta}] \;\}
\end{equation}
gives rise to the folding kernel
\begin{equation}\label{eqn3.27b}
  F(k,k',\kappa)=\delta(kk'-\kappa), \quad k,k' \in 
     [\sqrt{\alpha}, \sqrt{\beta}], \quad \kappa \in [\alpha, \beta] \;\}
\end{equation}
Therefore, the Gramians $A(k,k')$ and $B(k,k')$ are operators with 
$\kappa s'(\kappa)$ and $s(\kappa)$ on the hyperbolae $kk'=\kappa$. 
For constant weight $u=1$,
\begin{equation}\label{eqn3.27c}
   s(k) = \frac{k^b - k^a}{\ln k}, \quad 
k s'(k) = \frac{bk^b - ak^a}{\ln k} - \frac{s(k)}{\ln k}
\end{equation}
For a numerical experiment, we construct a Gaussian quadrature for 
$G(k,x)=k^x$, $x \in [a,b]=[-3,3]$, $k \in [\alpha, \beta]=[1/16, 4]$ 
by constructing an inner product quadrature for the factor space 
$T(k,x)= k^x$, $x \in [a,b]=[-3,3]$, $k \in [\sqrt{\alpha}, \sqrt{\beta}]
=[1/4, 2]$. For a cut off precision $\epsilon=10^{-12}$, the procedures of 
\S \ref{sec-3.5} gives rise to $n=9$. Figure \ref{fig-4.3} shows the 
locations of nodes in $[-3,3]$, and the relative error of the quadrature 
as a function of $k \in [1/16, 4]$. 
\begin{figure}[h]
\epsfig{file=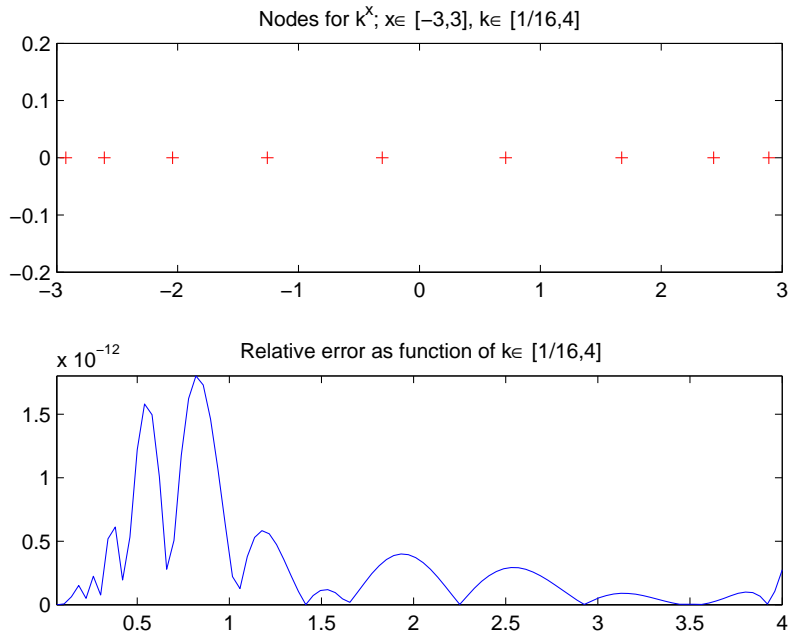, width=4.9in}
\caption{Quadrature nodes and relative error for $G(k,x)=k^x$}
\label{fig-4.3}
\end{figure}
%


 




\section{Generalizations and applications} \label{sec-5}

The algorithms for the inner product quadrature design, presented
in Theorem \ref{thrm-5}, \ref{thrm-5}, and \ref{thrm-3.3}, will also
work for matrix and tensor quadrature weights. Take the two Gramians
$A, B$ of Theorem \ref{thrm-3.3} for example, the product space 
$\Pi(S,T)$ of (\ref{eqn3.4x}) may have a dimension on the order 
$mn$. A quadrature of $r$ nodes, with $r \le m$, will integrate 
these $O(mn)$ distinct functions only if the quadrature weights 
$w$ has off diagonal entries. It may be a banded matrix, or a 
dense matrix with a predetermined diagonals, but as long as the 
$r$-by-$r$ matrix $w$ is invertible, Theorem \ref{thrm-3.3} still 
holds, for its proof is equally valid as the diagonal matrix $w$ 
is replace by an invertible one. 

Tensor ``weights'' refer to one or two $r$-by-$r$ matrix $w$ which 
will entrywise multiply the integrand $T_r$ from the right, or $S_r$ 
from the left, or both, as opposed to standard matrix-matrix 
multiplication. The proof of Theorem \ref{thrm-3.3} holds, as $T_r$ 
and $S_r$ will still be full rank $r$ after the entrywise 
multiplication, otherwise the rank of $B$ will be less than $r$.

Matrix and tensor quadrature weights are related to certain sensing 
and inverse scattering applications; see \S \ref{sec-5.1} for more 
details.

The 1-D results presented in this paper makes a step toward a 
systematic method to design Gaussian quadratures for an arbitrary 
system of functions in one and higher dimensions; see \S 
\ref{sec-5.2} and \ref{sec-5.3} for 2-D extensions.

\subsection{Separation principle of imaging} \label{sec-5.1}

The mathematical models for imaging, with the notable exceptions 
of MRI and X-ray CAT scan due to absence of wave scattering as 
their probing mechanisms, are inconsistent in that their 
formulation is based on reflectivity or scattering coefficient 
of targets as a function of position. But in many applications, 
these functions are not nearly single valued. Amplitude of backward, 
monostatic reflected wave from a small target depends on direction 
unless the target is a ball, for example, with uniform reflection 
coefficient on the sphere. 

There is a remarkable property of Gaussian quadrature design - the 
nodes can be determined \emph{first} and \emph{independently} of 
the weights. This is also valid for a "quadrature" with inconsistent 
``quadrature weights", namely with tensor weights. For imaging or 
inverse scattering with waves, the measurement is typically 
a Gramian matrix known as the scattering matrix. For some $r$
point targets as the scatterers, there is a $r$-term quadrature
to integrate the Gramian matrix, and the quadrature notes fall
on the locations of the point targets, provided that the size of
the Gramian matrix is no less than $r$. Thus, the quadrature 
approach presents an alternative model based on the locations
of targets.

If we construct a quadrature for the Gramian matrix, the locations
of the targets will be determined \emph{first} and \emph{separately} 
from the target's reflectivities, whether or not they are consistent.
If consistent, and if there is no multiple scattering among them 
then the quadrature weights will be the reflectivities; if there is
multiple scattering then the quadrature weights will be a dense matrix
which together with the quadrature nodes will be sufficient to 
recover the consistent reflectivities via solution of a simple 
matrix equation. If the reflectivities are inconsistent, the
quadrature weights will be tensor, and it is possible to assign
an average reflectivity to each point target.

\subsection{Quadratures in higher dimensions}
\label{sec-5.2}

A Gaussian quadrature in two dimensions integrating the bivariate 
polynomials of degree less than $2n$ in a domain $D$, as is well
known, is a summation of $n(n+1)/2$ terms. Such a quadrature 
rarely exits. We will, however, provide a 2-D versions of Theorems 
\ref{thrm-1} and \ref{thrm-2} to construct the quadrature by eigen 
decomposition, and to illustrate what is required of quadrature 
design in higher dimensions. The results will also be useful in 
\S \ref{sec-5.3} for quadrature in two and higher dimensions 
constructed by a technique called deflation.  

Let $T(n,x,y)$ of size $n(n+1)/2$-by-$D$ be the $n(n+1)/2$ basis
functions for polynomials of degree less than $n$ in the domain $D$. 
Let
\begin{eqnarray}
  B &=& T(n,x,y) \cdot T(x,y,n)  \label{eqn2.5b0} \\
  A_x &=& T(n,x,y) \cdot x T(x,y,n)  \label{eqn2.5b} \\
  A_y &=& T(n,x,y) \cdot y T(x,y,n)  \label{eqn2.5c}
\end{eqnarray}
where the dot product is over domain $D$ and with a weight function
$u$. We have
\begin{theorem}  \label{thrm-3}
If there is a $n(n+1)/2$-term quadrature $\{ (x_j,y_j); w_j \}$ to 
integrate the Gramian matrices  $B$, $A_x$, and $A_y$, then 
\begin{eqnarray}
  \lambda_j(A_x B^{-1}) &=& x_j, \quad j=1\!:\!n(n+1)/2 \label{eqn2.6a} \\
  \lambda_j(A_y B^{-1}) &=& y_j, \quad j=1\!:\!n(n+1)/2 \label{eqn2.6b}
\end{eqnarray}
\end{theorem}
Here the weight $u$ is not assumed positive. This result will be useful 
in \S \ref{sec-5.2} for deflating the Gramians.
\begin{theorem}  \label{thrm-4}
Let the weight function $u$ be positive definite, and let $v_j$ 
denote the j-th eigenvector of a matrix. The three conditions are 
equivalent \\
(i)
There is a $n(n+1)/2$-term quadrature $\{ (x_j,y_j); w_j \}$ to 
integrate $B$, $A_x$, and $A_y$.  \\
(ii) The two quotient matrices share common eigen space, and
\begin{eqnarray}
  \lambda_j(A_x B^{-1}) &=& x_j, \quad j=1\!:\!n(n+1)/2 \label{eqn2.8a} \\
  \lambda_j(A_y B^{-1}) &=& y_j, \quad j=1\!:\!n(n+1)/2 \label{eqn2.8b} \\
  v_j(A_x B^{-1}) &=& v_j(A_y B^{-1}) = T(n,x_j,y_j),  
       \quad j=1\!:\!n(n+1)/2  \label{eqn2.9a}
\end{eqnarray}
(iii) The $n+1$ orthogonal polynomials of degree $n$ have $n(n+1)/2$ 
real, pairwise distinct, common zeros $\{ (x_j,y_j), \; j=1\!:\!n(n+1)/2 \}$.
\end{theorem}
The proof of equivalency of (i) and (ii) is similar to that of Theorem
\ref{thrm-2}. For (iii), see the proof of Theorem \ref{thrm-6}.

\subsection{Deflation for 2-D quadrature design}
\label{sec-5.3}

A node of a 2-D quadrature provides 3 parameters $\{ (x_j,y_j); 
w_j \}$. Denote by $P^{(2)}_n$ the linear space of bivariate 
polynomials of degree less than $n$. Therefore,
\begin{equation}\label{eqn4.10f}
   \mbox{\bf dim}(P^{(2)}_n) = n(n+1)/2, \quad \mbox{and} \quad
   \mbox{\bf dim}(P^{(2)}_{2n}) = n(2n+1)
\end{equation}
A quadrature integrating $P^{(2)}_{2n}$ generally requires no less 
than a third as many nodes as the dimension, namely $n(2n+1)/3$ 
nodes.

A classical Gaussian quadrature for bivariate polynomials, if exists, 
can be constructed by Theorem \ref{thrm-3}, using $n(n+1)/2$ nodes to 
integrate $P^{(2)}_{2n}$; therefore, the quadrature problem is 
over-determined and rarely has a solution. The algorithm of Theorem 
\ref{thrm-3} is rarely useful. But it can be modified and made 
useful by deflating the Gramians $A$ and $B$ iteratively. 

The eigen decomposition of Theorem \ref{thrm-3} can only provide 
$n(n+1)/2$ nodes. Additional nodes will be determined by other 
mechanisms. The number of these nodes is 
\begin{equation}\label{eqn4.10c}
   dN = n(2n+1)/3 - n(n+1)/2 = n(n-1)/6
\end{equation}
which is about a third of $n(n+1)/2$, namely a third of what can 
be provided by the eigen decomposition. In 3-D, the ratio is 1; 
as many additional nodes are requires as those by the eigen 
decomposition. Deflation is a method to provide the additional 
nodes iteratively. The following description takes bivariate 
polynomials in a triangle as example to illustrate the method.
\begin{enumerate}
 \item Suppose that a total of 40 nodes are required to integrate
 polynomials of degree less than $2n$ for some $n$. Suppose that
 the size of Gramians is 30, so eigen decomposition can only provide
 30 nodes. Additional 10 nodes will be supplied by an iterative 
 procedure. 
 
 \item Suppose we are given the precise locations of 10 out of the 
 40 nodes and the corresponding weights $w_j$. Each node $z_j = 
 (x_j, y_j)$, $j=1\!:\!10$, gives rise to a rank one matrix 
 $T(n,z_j) w_j T(z_j,n)$;
 see Theorem \ref{thrm-3} for notation. Deflation involves three
 steps (i) Remove these 10 matrices from Gramian $B$ (ii) Remove the 
 10 rank one matrices $T(n,z_j) w_j x_j T(z_j,n)$ from Gramian 
 $A_x$ (iii) Remove the 10 rank one matrices $T(n,z_j) w_j y_j 
 T(z_j,n)$ from Gramian $A_y$.

 \item Theorem \ref{thrm-3b} below states that the eigen decomposition 
 on the quotient matrices $A_x B^{-1}$ and $A_y B^{-1}$ (cf Theorem 
 \ref{thrm-3}) after the deflations will provide the exact locations 
 of the remaining 30 nodes.

 \item Initialization. Choose 10 nodes and weights as initial guess. 
 There are ways to make good initial guess located in a corner 
 of the triangle - the domain of integration. 

 \item Iteration. Eigen decomposition of the deflated quotient 
 matrices to obtain 30 nodes. Discard 20 of them by choosing only 
 10 out of 30 that are farthest from the 10 initial guess, and
 use them as the initial guess for the next iteration.

 \item Convergence. The Coulomb potential $1/r$ decays over distance. 
  Its perturbation due to that of charge location decays faster: $1/r^2$.
  The location errors in the 10 initial guess will have minimal 
  influence on the farthest of the 30 nodes.
\end{enumerate}
Deflation is also useful for constructing (i) Gauss-Radau type 
formula (with an end x=a or b fixed as a quadrature node) in one 
and higher dimensions
(ii) Gauss-Lobatto type formula (with two ends fixed as quadrature 
nodes) in one and higher dimensions. 

Deflation can be used for constructing a Gaussian quadrature in 
a submain and merging it to an existing quadrature as the trapezoidal 
rule in another subdomainsuch - the so-called hybrid rules \cite{alpert2}.
\begin{theorem}  \label{thrm-3b}
Suppose there is a $n(n+1)/2+r$-term quadrature $\{ (x_j,y_j); w_j \}$ 
to integrate the Gramians $B$, $A_x$, and $A_y$ of 
(\ref{eqn2.5b0})-(\ref{eqn2.5c}). For the first $r$ nodes, let the 
deflated Gramians be defined by
\begin{eqnarray}
  \dot{B}  &=& B - \sum_{j=1}^r T(n,x_j,y_j) w_j T(x_j,y_j,n) \label{eqn2.5d} \\
  \dot{A}_x &=& A_x - \sum_{j=1}^r T(n,x_j,y_j) w_j x_j  T(x_j,y_j,n) 
           \label{eqn2.5e} \\
  \dot{A}_y &=& A_y - \sum_{j=1}^r T(n,x_j,y_j) w_j y_j  T(x_j,y_j,n) 
           \label{eqn2.5f}
\end{eqnarray}
then 
\begin{eqnarray}
  \lambda_j(\dot{A}_x \dot{B}^{-1}) &=& x_j, \quad j=1+r\!:\!n(n+1)/2+r 
              \label{eqn2.6j} \\
  \lambda_j(\dot{A}_y \dot{B}^{-1}) &=& y_j, \quad j=1+r\!:\!n(n+1)/2+r 
              \label{eqn2.6k}
\end{eqnarray}
\end{theorem}
The proof is a direct consequence of Theorem \ref{thrm-3} applied to the
deflated weight function
\begin{equation}
  \dot{u}(x) = u(x) - \sum_{j=1}^r w_j \delta(x-x_j, y-y_j) \label{eqn3.3h}
\end{equation}

\end{document}